\documentclass[a4paper]{amsart}
\usepackage{latexsym}\usepackage{ifthen}\usepackage[leqno]{amsmath}
\usepackage{enumerate}\usepackage{calc}\usepackage{hyphenat}
\usepackage{amstext,amsbsy,amsopn,amsthm,amsgen}
\usepackage{amsfonts,amscd,amsxtra,upref}
\usepackage{mathrsfs}\usepackage{euscript}\usepackage{amssymb}
\usepackage{hyperref}
\usepackage{color}
\swapnumbers
\theoremstyle{plain}
\makeatletter\@namedef{subjclassname@2020}{\textup{2020} Mathematics Subject Classification}
\makeatother
\newtheorem{Thm}{Theorem}[section]
\newtheorem{Lem}[Thm]{Lemma}
\newtheorem{Cor}[Thm]{Corollary}
\newtheorem{Pro}[Thm]{Proposition}

\theoremstyle{definition}
\newtheorem{Def}[Thm]{Definition}
\newtheorem{Exm}[Thm]{Example}

\newtheorem{Prb}[Thm]{Problem}
\theoremstyle{remark}
\newtheorem{Rem}[Thm]{Remark}
\numberwithin{equation}{section}

\newcommand{\ITE}[3]{\ifthenelse{#1}{#2}{#3}}\newcommand{\ITEE}[4][]{\ITE{\equal{#2}{#3}}{#4}{#1}}
 \ITE{\isundefined{\texorpdfstring}}{\newcommand{\texorpdfstring}[2]{#1}}{}

\newcommand{\myData}[1][]{
 \author[P.\ Niemiec]{Piotr Niemiec}
 \address{\ITEE{#1}{*}{P.\ Niemiec{}\\}%%
  Wydzia\l{} Matematyki i~Informatyki\\Uniwersytet Jagiello\'{n}ski\\%%
  ul.\ \L{}o\-ja\-sie\-wi\-cza 6\\30-348 Krak\'{o}w\\Poland}
 \email{piotr.niemiec@uj.edu.pl}
 }

\newenvironment{cor}[2][]{\ITEE[{\begin{Cor}[#1]}]{#1}{}{\begin{Cor}}\label{cor:#2}}{\end{Cor}}
\newenvironment{dfn}[2][]{\ITEE[{\begin{Def}[#1]}]{#1}{}{\begin{Def}}\label{def:#2}}{\end{Def}}
\newenvironment{exm}[2][]{\ITEE[{\begin{Exm}[#1]}]{#1}{}{\begin{Exm}}\label{exm:#2}}{\end{Exm}}
\newenvironment{lem}[2][]{\ITEE[{\begin{Lem}[#1]}]{#1}{}{\begin{Lem}}\label{lem:#2}}{\end{Lem}}
\newenvironment{prb}[2][]{\ITEE[{\begin{Prb}[#1]}]{#1}{}{\begin{Prb}}\label{prb:#2}}{\end{Prb}}
\newenvironment{pro}[2][]{\ITEE[{\begin{Pro}[#1]}]{#1}{}{\begin{Pro}}\label{pro:#2}}{\end{Pro}}
\newenvironment{rem}[2][]{\ITEE[{\begin{Rem}[#1]}]{#1}{}{\begin{Rem}}\label{rem:#2}}{\end{Rem}}
\newenvironment{thm}[2][]{\ITEE[{\begin{Thm}[#1]}]{#1}{}{\begin{Thm}}\label{thm:#2}}{\end{Thm}}

\newcommand{\COR}[2][!]{\ITEE{#1}{!}{Corollary~}\ITEE{#1}{s}{Corollaries~}\textup{\ref{cor:#2}}}
\newcommand{\DEF}[2][!]{\ITEE{#1}{!}{Definition~}\ITEE{#1}{s}{Definitions~}\textup{\ref{def:#2}}}
\newcommand{\EXM}[2][!]{\ITEE{#1}{!}{Example~}\ITEE{#1}{s}{Examples~}\textup{\ref{exm:#2}}}
\newcommand{\LEM}[2][!]{\ITEE{#1}{!}{Lemma~}\ITEE{#1}{s}{Lemmas~}\textup{\ref{lem:#2}}}
\newcommand{\PRO}[2][!]{\ITEE{#1}{!}{Proposition~}\ITEE{#1}{s}{Propositions~}\textup{\ref{pro:#2}}}
\newcommand{\REM}[2][!]{\ITEE{#1}{!}{Remark~}\ITEE{#1}{s}{Remarks~}\textup{\ref{rem:#2}}}
\newcommand{\THM}[2][!]{\ITEE{#1}{!}{Theorem~}\ITEE{#1}{s}{Theorems~}\textup{\ref{thm:#2}}}

\newcommand{\CCC}{\mathbb{C}}
\newcommand{\KKK}{\mathbb{K}}
\newcommand{\RRR}{\mathbb{R}}
\newcommand{\TTT}{\mathbb{T}}
\newcommand{\ZZZ}{\mathbb{Z}}
\newcommand{\PpP}{\EuScript{P}}
\newcommand{\QqQ}{\EuScript{Q}}
\newcommand{\eE}{\mathfrak{e}}

\newcommand{\dd}{\colon}
\newcommand{\df}{\stackrel{\textup{def}}{=}}
\newcommand{\epsi}{\varepsilon}
\newcommand{\varempty}{\varnothing}

\newcommand{\OPN}[1]{\operatorname{#1}}
\newcommand{\dist}{\operatorname{dist}}
\newcommand{\RE}{\operatorname{Re}}
\newcommand{\UP}[1]{\textmd{\textup{#1}}}

\newcommand{\tfcae}{the following conditions are equivalent:}

\newcommand{\MBOX}[1]{\mbox{\(#1\)}}
\newcommand{\HOM}[2]{\MBOX{\OPN{Hom}_c(#1,#2)}}
\newcommand{\NORM}[1][]{value#1}
\newcommand{\reg}[1]{#1_{\UP{reg}}}
\newcommand{\REG}[1][Q]{\OPN{R#1V}}
\newcommand{\DUAL}[1][m]{#1etric-dual}
\newcommand{\SECD}[1]{\mbox{\(\widehat{\mbox{\(\widehat{#1}\)}}\)}}
\newcommand{\fin}{\OPN{bd}}
\newcommand{\zeroqv}[1]{\OPN{zero}_{#1}}
\newcommand{\inftyqv}[1]{\OPN{infty}_{#1}}

\begin{document}

\title{Metric duality for Abelian groups}
\myData\thanks{Research supported by the National Center of Science, Poland
 under the Weave-UNISONO call in the Weave programme [grant no
 2021/03/Y/ST1/00072].}
\begin{abstract}
The main aim of the paper is to introduce the concept of metric duality in
the category of topological Abelian groups that extends the classical notion of
duality for normed vector spaces and behaves quite nicely for LCA groups
(equipped with \textit{nice} metrics). In particular, it is shown that each
Polish LCA group admits a reflexive proper metric and, more generally, all LCA
groups possess reflexive (proper) metric structures.
\end{abstract}
\maketitle

%\thisPaper

\section{Introduction}

Duality appears quite naturally in categories of locally compact Abelian groups
and of Banach spaces. It is not only an interesting notion but also an important
tool. Although in both these realms reflexitivy is defined in a quite similar
way, due to our best knowledge so far no efforts have been made to explain that
these two different concepts (of both duality and reflexivity) actually arise
from a more general notion that leads to the Pontryagin's dual group for LCA
groups and to the dual Banach space for (real or complex) normed vector spaces.
In this paper we propose such a concept and present most fundamental properties
of this new notion. We call it \emph{metric duality} to distinguish it from
a classical duality for topological Abelian groups (with a special emphasis on
almost periodic groups; the reader interested in this topic is referred to,
e.g., \cite{cdm}) which we call \emph{topological duality}. The main differences
between these two ideas are the following:
\begin{itemize}
\item the metric dual of a metric Abelian group is metric as well, which, in
 general, fails to hold for the topological dual;
\item all Banach spaces are topologically reflexive (as topological groups)
 \cite{smi}, whereas the metric dual of a normed vector space is naturally
 isometrically isomorphic to its ordinary dual Banach space;
\item the topology of the metric dual \textbf{does} depend on the compatible
 metric on a given metrisable group~--- in particular, the metric dual of
 a Polish metric LCA group topologically coincides with its Pontryagin dual
 group iff its metric is proper (see \THM{comp-dual} in Section~3 below).
\end{itemize}
We consider our concept of metric duality as a natural generalisation of duality
in Banach spaces (not in LCA groups). However, in a fully general context (that
is, when dealing with metric Abelian groups other than normed vector spaces)
a new phenomenon appears (that is not present in normed vector spaces):
the canonical homomorphism from a given group into its second dual group may not
be isometric (even in case it is a topological embedding). This strange
behaviour leads to the notions of \emph{regularity} and \emph{semireflexivity}
of metrics.\par
It is also worth underlying that, in general, although the metric dual of
a Banach space \(E\) naturally `coincides' with the ordinary dual Banach space
of \(E\), the weak topologies on \(E\) induced by these two duals differ (simply
because the weak topology of the LCA group \((\RRR,+)\) does not coincide with
its natural topology).\par
The paper is organised as follows. In Section~2 we define \NORM[s] on Abelian
groups (they are single variable functions that naturally correspond to
invariant metrics) and their dual \NORM[s] (defined on dual groups). We
distinguish regular, semiregular, reflexive and semireflexive \NORM[s]. In
\THM{NVS} we show that the metric dual group of a normed vector space is
naturally isometrically isomorphic to the ordinary dual Banach space. In
particular, a Banach space is reflexive as a metric group iff it is so as
a Banach space. In the next section we deal with Polish LCA groups and prove (in
\THM{comp-dual}) that the dual \NORM{} of a compatible \NORM{} \(p\) on such
a group \(G\) is compatible with the Pontryagin's topology of the dual group iff
\(p\) is proper. As a consequence, we show (in \PRO{1-1LCA}) that there is
a natural one-to-one correspondence between reflexive proper \NORM[s] on \(G\)
and such \NORM[s] on the dual group of \(G\). In the fourth part we establish
fundamental properties of reflexive metric Abelian groups. In particular, in
\THM{main-refl} we show that two reflexive equivalent \NORM[s] on an Abelian
group have equivalent dual \NORM[s] and that metrically reflexive groups are
topologically reflexive. It is also shown that the metric dual group of
a reflexive Polish metric group is Polish as well. Section~5 is devoted to
a generalisation of the notions introduced in the second part to the context of
(so-called) quasi-\NORM[s] on Abelian groups. This material may be considered
preparatory to introducing the concept of metric duality for non-metrisable
topological Abelian groups (e.g., for all LCA groups). We establish there basic
properties of regular quasi-\NORM[s] (among other things, we show in
\THM{sup-reg} that all such quasi-\NORM[s] on a given Abelian group form
a complete lattice; and in \THM{sub-reg} that the restriction of a regular
quasi-\NORM{} to an arbitrary subgroup is regular as well). In Section~6 we
define metric structures on both abstract and topological Abelian groups and
introduce and study (metric) duality and reflexivity for groups with such
structures. We conclude this part by proving (in \THM{LCA-metr-str}) that each
LCA group admits a reflexive proper metric structure. In the final, seventh part
we pose some problems related to the topics discussed in earlier parts.

\subsection{Notation and terminology}
In this paper \(\RRR_+ = [0,\infty)\) and \(\TTT = \{z \in \CCC\dd\ |z| = 1\}\).
The (multiplicative) circle group \(\TTT\) shall be sometimes identified with
the additive group \(\RRR / \ZZZ\).\par
All topological groups are Hausdorff, unless otherwise stated. The neutral
element of a group \(G\) is denoted by \(e_G\). A \textit{\UP[semi\UP]\NORM} on
\((G,\cdot)\) is a function \(p\dd G \to \RRR_+\) such that the formula
\(G \times G \ni (x,y) \mapsto p(x^{-1} y) \in \RRR_+\) defines a [pseudo]metric
on \(G\). (More often authors use the term a \textit{norm} instead of a \NORM,
but in this paper we will also deal with normed vector spaces and therefore have
chosen the latter term to avoid confusions.) By a \textit{metric} group we mean
a triple \((G,\cdot,p)\) where \((G,\cdot)\) is a metrisable topological group
and \(p\dd G \to \RRR_+\) is a compatible \NORM{} on \(G\). (In this paper
almost all groups under consideration are Abelian.) \textit{LCA} is
the abbreviation of \textit{locally compact Abelian}. For a LCA group \(G\),
the compact-open topology on its dual group (that is, on the group of all
continuous homomorphisms from \(G\) into \(\TTT\)) is called by us
\emph{Pontryagin's topology}.\par
For any semi\NORM{} \(p\) on a group \(G\) and a positive real number \(r\) we
use \(B_p(r)\) and \(\bar{B}_p(r)\) to denote the open and closed \(p\)-balls
around \(e_G\); that is, \(B_p(r)\) [resp. \(\bar{B}_p(r)\)] consists of all
\(g \in G\) such that \(p(x) < r\) [resp. \(p(x) \leq r\)].\par
For two topological groups \(G\) and \(H\), \HOM{G}{H} stands for the set of all
continuous homomorphisms from \(G\) into \(H\). When, in addition, \(H\) is
Abelian, \HOM{G}{H} becomes a group with (natural) pointwise binary action. We
underline that, unless otherwise stated, \HOM{G}{H} is equipped with no
topology.

\section{\DUAL[M] group}

Before passing to metric LCA groups, first we discuss the concept of metric
duality in a more general context. The lemma below is a key tool of the whole
paper (and a counterpart of the analogous result for linear operators between
normed vector spaces). It is likely it is known. However, we did not find it
in the existing literature.

\begin{lem}{key}
Let \(q\) be a \NORM{} on \((\TTT,\cdot)\) such that
\begin{equation}\label{eqn:key}
\frac{|z-1|}{M} \leq q(z) \leq M |z-1|
\end{equation}
for all \(z \in \TTT\) and some constant \(M > 1\). For any metric group
\((G,\cdot,p)\) and a homomorphism \(\phi\dd G \to \TTT\) \tfcae
\begin{enumerate}[\upshape(i)]
\item \(\phi\) is continuous;
\item there exists a constant \(c \geq 0\) such that
 \begin{equation}\label{eqn:Lip}
 q(\phi(x)) \leq c p(x) \qquad (x \in G).
 \end{equation}
\end{enumerate}
\end{lem}

It is worth noting here that condition \eqref{eqn:key} is equivalent to
the statement that \(q\) induces a Gleason metric on the Lie group \(\TTT\) (for
the definition and more on Gleason metrics the reader is referred to Chapter~1
of \cite{tao}, with a special focus on Defintion~1.3.8 and Theorem~1.5.5
therein). The proof below shows that the conclusion of \LEM{key} remains valid
if \(\TTT\) equipped with \(q\) is replaced by a (possibly non-Abelian) Lie
group equipped with a Gleason metric.

\begin{proof}[Proof of \LEM{key}]
We only need to show that (ii) follows from (i). To this end, first note that it
follows from \eqref{eqn:key} that there exists a constant \(\epsi \in (0,1)\)
such that
\begin{itemize}
\item[(\(\star\))] If \(z \in \TTT\) and \(n > 0\) are such that \(q(z^j) \leq
 \epsi\) for \(j=1,\ldots,n\), then \(q(z^n) \geq \epsi n q(z)\).
\end{itemize}
Now take a positive real number \(r > 0\) such that \(\phi(B_p(r)) \subset
B_q(\epsi)\). We will show that \eqref{eqn:Lip} holds for \(c \df
\frac{4M}{\epsi r}\). So, fix \(x \in G\). We may and do assume that \(\phi(x)
\neq 1\). If \(x \notin B_p(r)\), then \(q(\phi(x)) \leq M |\phi(x)-1| \leq 2M
\leq \frac{2M}{r} p(x)\). On the other hand, if \(p(x) < r\), take the greatest
integer \(N > 0\) such that \(N p(x) < r\). Then \(p(x) \geq \frac{r}{2N}\),
\(q(\phi(x)^j) \leq \epsi\) for \(j=1,\ldots,N\) (because \(x^j \in B_p(r)\))
and hence, thanks to (\(\star\)), \(q(\phi(x)) \leq \frac{q(\phi(x)^N)}{N \epsi}
\leq \frac{2M}{N\epsi} \leq \frac{4M}{r\epsi} p(x)\) and we are done.
\end{proof}

In the sequel we will define the dual \NORM{} (to a given one on a group \(G\))
as a function that assigns to a character \(\phi \in \HOM{G}{\TTT}\) the least
constant \(c\) in \eqref{eqn:Lip}. But first we should choose a certain `nice'
\NORM{} on \(\TTT\) that will be fixed in the whole theory. Actually, each
\NORM{} on \(\TTT\) that satisfies \eqref{eqn:key} may serve as such a \NORM,
but then not all results presented in this paper remain valid (e.g., explicit
formulas for the dual \NORM{} or the regularisation of a given \NORM{} can
become incorrect). To justify our choice, we introduce the following (well-known
and quite easy)

\begin{pro}{great-circ}
Let \(\lambda\) denote the quotient \NORM{} on \(\TTT \equiv \RRR / \ZZZ\)
induced by the natural one on \(\RRR\); that is:
\begin{equation}\label{eqn:lambda}
\lambda(z) = \min \{|t|\dd\ t \in \RRR,\ e^{2\pi ti} = z\}.
\end{equation}
For a compatible \NORM{} \(p\) on \(\TTT\) \tfcae
\begin{itemize}
\item for any \(z \in \TTT \setminus \{1\}\) there exists an isometric path
 from \([0,p(z)]\) into \((\TTT,p)\) that joins \(1\) with \(z\);
\item \(p = c \lambda\) for some constant \(c > 0\).
\end{itemize}
\end{pro}

We leave the proof of \PRO{great-circ} to the reader. The conclusion of this
result lets us call \NORM[s] of the form \(\lambda_c \df c \lambda\)
\textit{canonical}. (Note that the diameter of the space \((\TTT,\lambda_c)\) is
\(\frac{c}{2}\).)\par
\textbf{From now on, the group \boldmath{\(\TTT\)} is equipped with the \NORM{}
\boldmath{\(\lambda_c\)}} where \(c > 0\) is some fixed constant. (We will set
\(c\) as 1 after \PRO{reals}.) Note that \(\lambda_c\) satisfies
\eqref{eqn:key}.

\begin{dfn}{dual-value}
\emph{Dual \NORM} to a \NORM{} \(p\) on an Abelian group \((G,+)\) is a function
\(p'\dd \HOM{G}{\TTT} \to \RRR_+\) (where \(G\) is topologised by \(p\)) given
by the formula:
\[p'(\chi) \df \inf \{M > 0\dd\ \lambda_c(\chi(x)) \leq M p(x)\ (x \in G)\}
\qquad (\chi \in \HOM{G}{\TTT}).\]
The triple \((G',\cdot,p')\) where \(G' = \HOM{G}{\TTT}\) is called
the \emph{\DUAL{} group} of \((G,+,p)\). In other words, \((G,+,p)' \df
(G',\cdot,p')\).
\end{dfn}

To complete the above definition, we recall that every Abelian group becomes
topological when equipped with a \NORM{} (or, more precisely, with an invariant
metric induced by a \NORM). It can also easily be checked that the dual \NORM{}
of any \NORM{} is again a \NORM{}. In this way we can define inductively \DUAL{}
groups of higher order by \((G,+,p)^{(n)} = (G^{(n)},\cdot,p^{(n)}) \df
(G',\cdot,p')^{(n-1)}\).\par
It follows from the definition of a dual \NORM{} that for any \NORM{} \(p\) on
an Abelian group \(G\):
\[\lambda_c(\chi(x)) \leq p'(\chi) p(x) \qquad (x \in G,\ \chi \in G').\]
Here, at the very preliminary stage of our considerations, we underline that, in
general, the topology of the \DUAL{} group of a metric LCA group may differ
from the Pontryagin's topology of the dual group. However, both these dual
groups always coincide as sets with binary actions. A complete characterisation
of \NORM[s] on metrisable LCA groups whose dual \NORM[s] are compatible with
the Pontryagin's topology of the dual group will be given in \THM{comp-dual}
below.

\begin{exm}{T-Z}
Considering (in a natural and standard manner) the topological group \(\ZZZ\) as
the Pontryagin's dual of \(\TTT\) and vice versa, it can be easily checked that
the \DUAL{} group of \((\TTT,\cdot,\lambda_c)\) is \((\ZZZ,+,m)\) where \(m(k) =
|k|\). Conversely, the \DUAL{} group of \((\ZZZ,+,m)\) is
\((\TTT,\cdot,\lambda_c)\). The details are left to the reader.
\end{exm}

To fix uniquely the \NORM{} on \(\TTT\), we establish the following easy
property.

\begin{pro}{reals}
Let the quotient epimorphism of \(\RRR\) onto \(\RRR / \ZZZ \equiv \TTT\) be
realised as \(p\dd \RRR \ni t \mapsto e^{2\pi ti} \in \TTT\). For any \(t \in
\RRR\) let \(\phi_t \in \RRR'\) be given by \(\phi_t(x) = p(tx)\). (In this way
the assignment \(\RRR \ni t \mapsto \phi_t \in \RRR'\) defines an isomorphism of
topological groups.)\par
For any compatible \NORM{} \(\rho\) on \(\RRR\) let \(\rho^*\) be the \NORM{} on
\(\RRR\) corresponding to \(\rho'\) via the above isomorphism; that is,
\(\rho^*(t) = \rho'(\phi_t)\). Then, for \(\tau(t) \df |t|\):
\[\tau^* = c \tau.\]
\end{pro}
\begin{proof}
It follows from \eqref{eqn:lambda} that \(\lambda_c(\phi_t(x)) \leq c |tx|\)
and hence \(\tau^*(t) \leq c |t|\). On the other hand, when \(t\) is fixed and
\(x \neq 0\) is very close to \(0\), then \(\lambda_c(\phi_t(x)) = c |tx|\),
which yields the reverse inequality and finishes the proof.
\end{proof}

Taking into account the above result, it is quite natural to require
``axiomatically'' that the natural \NORM{} on the group \(\RRR\) is
\textit{self-dual} (after identifying this group with its dual via isomorphism
described in the above proposition). Hence we set \(c = 1\). (Note that this is
the case when the metric length of \(\TTT\) is 1, which best fits to
the probabilistic Haar measure on \(\TTT\).)\par
\textbf{From now on to the end of the paper, the group \boldmath{\(\TTT\)} is
equipped with the \NORM{} \boldmath{\(\lambda\)}} (given by
\eqref{eqn:lambda}).\par
As it is done for Banach spaces as well as for LCA groups, for any \NORM{} \(p\)
on an Abelian group \((G,+)\) we have a \emph{canonical} homomorphism
\(\kappa_G\dd (G,+,p) \ni x \to \eE_x \in (G,+,p)''\):
\begin{equation}\label{eqn:bidual}
\eE_x(\phi) = \phi(x) \qquad (x \in G,\ \phi \in G').
\end{equation}
Observe that \(p''(\eE_x) \leq p(x)\) and therefore \(\kappa_G\) is continuous
(even non-expansive). However, it is one-to-one iff \(G\) is maximally almost
periodic (cf. \cite{cdm}); that is, if \(G'\) separates the points of \(G\). In
particular, \(\kappa_G\) may be non-isometric. Taking these remarks into
account, we introduce

\begin{dfn}{reg-refl}
A \NORM{} \(p\) on an Abelian group \((G,+)\) (as well as the metric group
\((G,+,p)\)) is said to be
\begin{itemize}
\item \emph{regular} if \(\kappa_G\) is isometric;
\item \emph{semiregular} if \(\kappa_G\) is a topological embedding;
\item \emph{reflexive} if \(\kappa_G\) is both isometric and surjective;
\item \emph{semireflexive} if \(\kappa_G\) is a topological isomorphism.
\end{itemize}
\emph{Regularisation} of \(p\) is a semi\NORM{} \(\reg{p}\) on \(G\) defined by
\(\reg{p}(x) \df p''(\eE_x)\ (\leq p(x))\). (In particular, \(p = \reg{p}\) iff
\(p\) is regular; and \(p\) is semiregular iff \(\reg{p}\) is a \NORM{}
equivalent to \(p\).) We use \(\REG[](G)\) to denote the collection of all
regular \NORM[s] on \(G\).
\end{dfn}

\begin{pro}{regular}
Let \((G,+,p)\) be a metric Abelian group. Then:
\begin{enumerate}[\upshape(a)]
\item for any \(x \in G\),
 \[\reg{p}(x) = \sup_{\chi \in G' \setminus \{1\}}
 \left(\inf_{g \notin \ker(\chi)}
 \frac{\lambda(\chi(x))}{\lambda(\chi(g))} p(g)\right)\]
 (where \(\sup(\varempty) = 0\));
\item \(p^{(n)} \in \REG[](G^{(n)})\) for all \(n > 1\);
\item if \(p\) is semiregular, then \(\reg{p} \in \REG[](G)\) and \((\reg{p})' =
 p'\).
\end{enumerate}
\end{pro}

It is worth noting here that in Section~\ref{sec:quasi} we will extend
the notions of regularity and regularisations to (so-called) quasi-\NORM[s] and
prove that the regularisation is always regular (so, the conclusion of item (c)
above holds without any additional assumptions on \(p\)).

\begin{proof}[Proof of \PRO{regular}]
Part (a) is left to the reader. We turn to (b). For transparency, set \(q =
p'\). We have \(\lambda(\xi(\chi)) \leq q(\chi) q'(\xi)\) for all \(\chi \in
G'\) and \(\xi \in G''\). Substituting \(\xi = \eE_x\) (with \(x \in G\)), we
get
\begin{equation}\label{eqn:aux2}
\lambda(\chi(x)) \leq \reg{p}(x) p'(\chi).
\end{equation}
Applying the last inequality to \(q\) instead of \(p\) yields
\(\lambda(\xi(\chi)) \leq \reg{q}(\chi) q'(\xi)\) for all \(\chi \in G'\) and
\(\xi \in G''\). Again, substituting \(\xi = \eE_x\), we obtain
\(\lambda(\chi(x)) \leq \reg{p}(x) \reg{q}(\chi) \leq p(x) \reg{q}(\chi)\).
Since \(q(\chi)\) is the least Lipschitz constant of \(\chi\), we conclude that
\(q(\chi) \leq \reg{q}(\chi)\). So, \(q = \reg{q}\) and an easy induction
argument finishes the proof of (b).\par
Finally, to show (c), denote \(\rho = \reg{p}\). Since \(\rho \leq p\), we get
\(p' \leq \rho'\) (here we use the property that \(\rho\) is equivalent to
\(p\)). On the other hand, \eqref{eqn:aux2} implies that \(\rho' \leq p'\) and,
consequently, \(\rho' = p'\) and \(\rho'' = p''\), which is followed by (c).
\end{proof}

The above result shows that for any \NORM{} \(p\) on an Abelian group
\((G,\cdot)\), the sequence \((G',\cdot,p')\), \((G'',\cdot,p'')\), \ldots
behaves similarly as the classical sequence of consecutive dual spaces of
a Banach space in the sense that each of these groups embeds isometrically into
its second dual in a canonical way. We leave it as a warm-up exercise to give
an example of a non-regular metric Abelian group. That this cannot happen for
normed vector spaces is shown in the following

\begin{thm}{NVS}
A norm on a real or complex vector space is a regular \NORM. Moreover, for any
normed vector space \((E,\|\cdot\|)\) over the field \(\KKK \in \{\RRR,\CCC\}\)
the assignment
\begin{equation}\label{eqn:aux3}
E^* \ni u \mapsto p \circ (\RE u) \in E'
\end{equation}
defines an isometric group isomorphism from the dual Banach space \(E^*\) of
\(E\) onto the \DUAL{} group \(E'\) of \(E\) where \(p\dd \RRR \to \TTT\) is
given by \(p(t) = e^{2\pi ti}\).\par
In particular, \(E\) is a reflexive metric group iff it is a reflexive Banach
space.
\end{thm}
\begin{proof}
It is well-known that (since \(E\) is a simply connected topological group) for
any character \(\chi \in E'\) there exists a unique continuous \(\RRR\)-linear
functional \(\phi\dd E \to \RRR\) such that \(p \circ \phi = \chi\). If \(\KKK =
\CCC\), any continuous \(\RRR\)-linear functional \(\phi\dd E \to \RRR\) extends
to a unique \(\psi \in E^*\) such that \(\|\psi\| = \|\phi\|\). So, to show that
\eqref{eqn:aux3} is isometric and surjective, we only need to verify that
the Lipschitz constants of \(\phi\) and \(p \circ \phi\) coincide for any
\(\RRR\)-linear functional \(\phi\dd E \to \RRR\). But the last property easily
follows from the fact that \(p\) is non-expansive (w.r.t. \(\lambda\) and
the natural \NORM{} on \(\RRR\)) and \(p(\phi(tx)) = |\phi(tx)|\) for any \(x
\in E\) and sufficiently small \(t > 0\).\par
For simplicity, denote by \(\Lambda_E\) the isomorphism given by
\eqref{eqn:aux3}, and by \(J_E\dd E \to E^{**}\) the canonical embedding of
\(E\) into its second dual Banach space. Observe that then \(\Lambda_{E^*}\dd
E^{**} \to (E^*)'\) and \(\Phi\dd E^{**} \ni \xi \mapsto \Lambda_{E^*}(\xi)
\circ \Lambda_E^{-1} \in E''\) is a well defined isometric group isomorphism
such that \(\kappa_E = \Phi \circ J_E\). Therefore \(\kappa_E\) is isometric,
and \(E\) is reflexive as a normed vector space iff it is so as a metric group.
\end{proof}

The above result shows that the concept of metric duality introduced in
\DEF{dual-value} naturally extends duality for normed vector spaces.

\begin{exm}{discrete}
Consider the discrete \NORM{} \(\delta_G\) on an Abelian group \((G,\cdot)\)
(that is, \(\delta_G(g) = 1\) for \(g \neq e_G\)). Then \(\delta_G'(\chi) =
\sup_{g \in G} \lambda(\chi(g))\) for any homomorphism \(\chi\dd G \to \TTT\).
In particular:
\begin{itemize}
\item if \(\chi\) has infinite or (finite) even order, then \(\delta_G'(\chi) =
 \lambda(-1) = \frac12\);
\item if \(\chi\) has odd order \(k\), then \(\delta_G'(\chi) =
 \frac{k-1}{2k}\).
\end{itemize}
In particular, \(\delta_G'\) induces the discrete topology on \(G'\). So, if
\(G\) is infinite, \(\delta_G'\) is not a compatible \NORM{} on the Pontryagin's
dual of (the discrete group) \(G\).\par
To simplify notation, denote \(q \df \reg{(\delta_G)}\). We know that \(q \leq
\delta_G\). If \(g \in G\) has infinite or (finite) even order, then there is
\(\chi \in G'\) such that \(\chi(g) = -1\) and therefore \(q(g) \geq
\frac{\lambda(\chi(g))}{\delta_G'(\chi)} = 1\), which shows that \(q(g) = 1\)
for such \(g\). On the other hand, if \(g \neq e_G\) has odd order \(k\), take
\(\chi \in G'\) such that \(\chi(g) = e^{2\pi \frac{k-1}{2k}i}\) to conclude
that \(q(g) \geq \frac{\lambda(\chi(g))}{\delta_G'(\chi)} \geq \frac23\). This
shows that \(q\) and \(\delta_G\) are Lipschitz equivalent and, in particular,
\(\delta_G\) is semiregular. However, in general, \(\delta_G\) may be
non-regular. For example, if \(G\) contains an element \(g\) of order \(9\),
then \(q(g^3) \neq 1\). Indeed, if \(\chi \in G'\) has order 3, then \(\chi(g^3)
= 1\). On the other hand, if the order of \(\chi\) differs from 3, then
\(\delta_G'(\chi) \geq \frac25\) and \(\lambda(\chi(g^3)) \leq \frac13\). So,
\(\frac{\lambda(\chi(g^3))}{\delta_G'(\chi)} \leq \frac56\) and, consequently,
\(q(g^3) \leq \frac56\) as well.
\end{exm}

Our nearest goal is to generalise fundamental properties of dual Banach spaces
to the context of \DUAL{} groups. Adapting terminology from Banach space theory,
for every metric Abelian group \(G\) we call the pointwise convergence topology
of \(G'\) the \emph{weak*} topology. The part (C) of the result below is
a counterpart of the classical Banach-Alaoglu theorem for Banach spaces.

\begin{thm}{dual-topo}
Let \(p\) be an arbitrary \NORM{} on an Abelian group \((G,+)\).
\begin{enumerate}[\upshape(A)]
\item The \NORM{} \(p'\) is complete.
\item The topology of \(G'\) induced by \(p'\) coincides with the topology of
 uniform convergence on \(p\)-bounded subsets of \(G\).
\item All closed \(p'\)-balls are compact in the weak* topology.
\end{enumerate}
\end{thm}
\begin{proof}
Part (A) is left as an exercise. To see (C), it is enough to observe that
\(\bar{B}_{p'}(r)\) is a closed subset (in the pointwise convergence topology)
of the set of all (possibly discontinuous) functions from \(G\) into \(\TTT\),
and to apply Tychonoff theorem. Here we focus only on (B). Since the topology of
uniform convergence on \(p\)-bounded subsets of \(G\) is metrisable, we shall
investigate convergence of ``usual'' sequences in both the topologies on
\(G'\).\par
First assume that
\begin{itemize}
\item[(\(\dag\))] characters \(\chi_1,\chi_2,\ldots \in G'\) converge to \(1\)
 uniformly on all \(p\)-bounded subsets of \(G\).
\end{itemize}
Fix \(\epsi \in (0,\frac14)\) and choose \(N\) so large that
\begin{equation}\label{eqn:aux4}
p(g) < \frac{1}{\epsi^2} \implies \lambda(\chi_n(g)) < \epsi \UP{ for all }
n \geq N.
\end{equation}
Now fix \(n \geq N\) and observe that \(\chi_n(B_p(\frac{1}{\epsi^2})) \subset
B_{\lambda}(\epsi)\). It follows from the proof of \LEM{key} that then
\(p'(\chi_n) \leq 4M \epsi\) where \(M > 0\) is a constant chosen so that
\eqref{eqn:key} holds for \(q = \lambda\). This shows that
\begin{equation}\label{eqn:aux5}
\lim_{n\to\infty} p'(\chi_n) = 0.
\end{equation}
Since \eqref{eqn:aux5} trivially implies (\(\dag\)), the proof is finished.
\end{proof}

Recall that the \emph{weak} topology on a topological Abelian group \(G\) is
the smallest topology on \(G\) with respect to which all characters from \(G'\)
are continuous. (In other words, a net \((x_{\sigma})_{\sigma\in\Sigma} \subset
G\) converges to \(g \in G\) in the weak topology iff \(\lim_{\sigma\in\Sigma}
\chi(x_{\sigma}) = \chi(g)\) for any continuous homomorphism \(\chi\dd G \to
\TTT\).) We underline that the weak topology on \(G\) may be non-Hausdorff (it
is Hausdorff iff \(G'\) separates the points of \(G\)).\par
It follows from the very definitions of the weak and weak* topologies that for
any reflexive \NORM{} \(p\) on an Abelian group \(G\) the canonical homomorphism
\(\kappa_G\dd G \to G''\) is a topological isomorphism when \(G\) and \(G''\)
are equipped with, respectively, the weak and the weak* topology. This remark
combined with \THM{dual-topo} yields

\begin{cor}{reflex-WK}
For any reflexive \NORM{} \(p\) on an Abelian group \((G,+)\) the following
properties hold.
\begin{enumerate}[\upshape(A)]
\item The \NORM{} \(p\) is complete.
\item The weak topology of \((G,+,p)\) is Hausdorff.
\item All closed \(p\)-balls in \(G\) are weakly compact.
\item The \NORM{} \(p'\) (on \(G'\)) is reflexive.
\end{enumerate}
\end{cor}

More on reflexive \NORM[s] the reader will find in Section~\ref{sec:refl-gen}.

\begin{rem}{functor}
In the theory of Banach spaces, to each continuous linear operator \(T\dd X \to
Y\) there naturally corresponds dual operator \(T^*\dd Y^* \to X^*\) given by
\(T^*(\phi) = \phi \circ T\), which is continuous as well. A similar concept of
duality for continuous homomorphisms between LCA groups is also well studied.
However, in metric duality (of metric Abelian groups) a strange phenomenon
occurs: although for each continuous homomorphism \(\rho\dd (G,+,p) \to
(H,+,q)\) the dual homomorphism \(\rho'\dd (H',\cdot,q') \to (G',\cdot,p')\) is
well defined (by \(\rho'(\chi) = \chi \circ \rho\)), in general it may be
discontinuous (see an example below). In this concept, a more appropriate
``arrows'' (using the language of category theory) between metric Abelian groups
(than just continuous homomorphisms) are homomorphism that are both continuous
and \emph{bounded} (that is, those that transform metrically bounded subsets of
the domain to such subsets of the target group). It may briefly be shown that if
\(\phi\dd (G,+,p) \to (H,+,q)\) is such a homomorphism, then \(\phi'\) is so as
well. (Continuity follows from the boundedness of \(\phi\) and from part (B) of
\THM{dual-topo}, whereas boundedness from equicontinuity of bounded sets in
\(H'\) and from the continuity of \(\phi\); cf. the proof of \LEM{key}).\par
Let us here give a short example of a continuous homomorphism whose dual is
discontinuous. To this end, denote by \(\delta_{\RRR}\), as in \EXM{discrete},
the discrete \NORM{} on \((\RRR,+)\) and let \(m\) and \(e\) be, respectively,
the natural \NORM{} on \((\RRR,+)\) and the identity homomorphism from
\((\RRR,+,\delta_{\RRR})\) into \((\RRR,+,m)\). It is clear that \(e\) is
continuous and that its dual homomorphism \(e'\), from \((\RRR,+,m)' \equiv
(\RRR,+,m)\) into \((\RRR,+,\delta_K)'\), is the identity as well. But then
\(e'\) is discontinuous, since its domain is connected and the target group is
discrete (consult \EXM{discrete}).
\end{rem}

\section{Metric duality for Polish LCA groups}

In this section we shall discuss in a greater detail metric duality for LCA
groups. We are mainly interested in such \NORM[s] on LCA groups \(G\) for which
the dual \NORM{} is compatible with the Pontryagin's topology of the dual group.
In particular, both the groups \(G\) and \(G'\) have to be metrisable. This is
the case precisely when \(G\) is Polish. That is why in the results of this part
we make such an assumption. Metric duality for general LCA groups will be
discussed in Section~\ref{sec:metr-str}.\par
The main tool of this section is a classical Glicksberg's theorem \cite{gli}
which reads as follows.

\begin{thm}{glick}
A weakly compact subset of a LCA group is compact.
\end{thm}

We begin with a simple

\begin{pro}{compact}
Let \(p\) be a compatible \NORM{} on a compact metrisable Abelian group \(G\).
Then:
\begin{enumerate}[\upshape(a)]
\item \(p\) is semireflexive;
\item \(p'\) is compatible with the Pontryagin's topology of \(G'\) (that is,
 \(p'\) induces the discrete topology);
\item \(\reg{p}\) and \(p'\) are reflexive.
\end{enumerate}
\end{pro}
\begin{proof}
First observe that \(\reg{p}\) is a \NORM. Since \(\reg{p} \leq p\) and \(G\) is
compact, \(\reg{p}\) is compatible. Consequently, \(p\) is semiregular and
\(\reg{p}\) is regular.\par
Further, there is a constant \(M > 0\) such that \(p(g) \leq M\) for any \(x \in
G\). Hence, \(p'(\chi) \geq \frac{1}{3M}\) for any non-constant \(\chi \in G'\),
because for each such \(\chi\) there is \(g \in G\) such that \(\lambda(\chi(g))
\geq \frac13\). This shows (b). In particular, \(\kappa_G\) is surjective, by
the Pontryagin's theorem. But then \(p\) is semireflexive and both \(\reg{p}\)
and \(p'\) are reflexive.
\end{proof}

In general Polish LCA groups situation is more subtle, as shown by

\begin{thm}{comp-dual}
The dual \NORM{} of a compatible \NORM{} \(p\) on a metrisable LCA group \(G\)
is compatible with the Pontryagin's topology of \(G'\) iff \(p\) is proper; that
is, if \(\bar{B}_p(r)\) is a compact set in \(G\) for each \(r > 0\). If \(p\)
is proper, it is semireflexive and \(\reg{p}\) and \(p'\) are both reflexive and
proper.\par
In particular, a Polish LCA group admits a compatible reflexive \NORM{} whose
dual \NORM{} is compatible with the Pontryagin's topology of the dual group.
\end{thm}
\begin{proof}
First assume \(p\) is a proper \NORM{} on a Polish LCA group. Since each compact
set in \(G\) is \(p\)-bounded, it follows from part (B) of \THM{dual-topo} that
\(p'\) is compatible with the Pontryagin's topology of \(G'\). In particular,
\((G',\cdot,p')\) is LCA, \(\kappa_G\) is a topological isomorphism (thanks to
the Pontryagin--van Kampen theorem) and the weak* topology of this group
coincides with the weak one. So, item (C) of \THM{dual-topo} combined with
\THM{glick} yields that \(p'\) is proper as well. Consequently, \(\reg{p}\) and
\(p'\) are reflexive (cf. \PRO{regular}) and \(p\) is semireflexive.\par
Now assume \(p\) is an arbitrary compatible \NORM{} on a metrisable LCA group
\(G\) such that \(p'\) is compatible with the Pontryagin's topology of \(G'\).
Then, again thanks to \THM{dual-topo} and \THM{glick}, \(p'\) is proper. So, it
follows from the first part of the proof that \(p''\) is compatible, proper and
reflexive. Consequently, \(p\) is semireflexive and \(\reg{p}\) is proper and
reflexive. But then, since \(\reg{p} \leq p\), also \(p\) is proper.\par
Now to conclude the proof it is sufficient to apply Struble's theorem \cite{str}
which asserts that each Polish locally compact group admits a proper compatible
left-invariant metric.
\end{proof}

The above result implies that a proper compatible \NORM{} on a LCA group is
reflexive iff it is regular.\par
The reader will easily prove the following

\begin{pro}{1-1LCA}
For any Polish LCA group \(G\) the assignment
\[p \leftrightarrow p'\]
establishes a one-to-one correspondence between reflexive proper compatible
\NORM[s] \(p\) on \(G\) and such \NORM[s] \(q = p'\) on \(G'\).
\end{pro}

\begin{rem}{in-the-large}
It is readily seen that for two reflexive proper compatible \NORM[s] \(p\) and
\(q\) on a Polish LCA group \(G\) the following equivalence holds:
\[p \leq q \iff q' \leq p'.\]
In particular, the concept of metric duality can be used to translate
the structure of \(G\) ``in the large'' [resp. ``in the small''] to
the structure of \(G'\) ``in the small'' [resp. ``in the large'']. This remark
seems to be particularly interesting when investigating metrisable compact
Abelian groups ``in the small'' or countable Abelian groups ``in the large.''
\end{rem}

The following result illustrates how the regularisation (introduced in
\DEF{reg-refl}) works. Recall that a function \(\omega\dd \RRR_+ \to \RRR_+\) is
said to be \emph{quasi-concave} (consult, e.g., Definition~5.6 on page 69 in
\cite{b-s}) if it is monotone increasing, vanishes at 0 and the function
\((0,\infty) \ni t \mapsto \frac{\omega(t)}{t} \in \RRR_+\) is monotone
decreasing. Continuous quasi-concave functions may be regarded as very ``nice''
moduli of continuity, as they automatically satisfy the inequality
\begin{equation}\label{eqn:tri-ineq}
\omega(x+y) \leq \omega(x)+\omega(y) \qquad (x, y \geq 0).
\end{equation}
Recall also that if \(\omega\dd \RRR_+ \to \RRR_+\) is a monotone increasing
function, then the assignment \(\RRR \ni x \mapsto \omega(|x|) \in \RRR_+\)
defines a proper compatible \NORM{} on \((\RRR,+)\) iff \(\omega\) is
continuous, vanishes only at \(0\), satisfies \eqref{eqn:tri-ineq} and
\(\lim_{t\to\infty} \omega(t) = \infty\).

\begin{pro}{q-c}
Let \(p\) be a proper compatible \NORM{} on \((\RRR,+)\) that is monotone
increasing on \(\RRR_+\). Under a canonical identification of \(\RRR'\) with
\(\RRR\) introduced in \UP{\PRO{reals}},
\begin{equation}\label{eqn:aux6}
p'(t) = \sup_{s \in (0,1/2]} \frac{s}{p(\frac{s}{t})} \qquad (t \in \RRR).
\end{equation}
In particular:
\begin{itemize}
\item \(p\) is reflexive iff \(p\restriction{\RRR_+}\) is quasi-concave;
\item if \(p\) is reflexive, then \(p'(t) = \frac{1}{2 p(1/(2t))}\).
\end{itemize}
\end{pro}
\begin{proof}
Observe that \(\lambda(e^{2\pi ti}) = \dist(t,\ZZZ)\) (where \(\dist(t,\ZZZ) =
\inf \{|t-k|\dd\ k \in \ZZZ\}\)) and therefore
\[p'(t) = \sup_{x>0} \frac{\dist(tx,\ZZZ)}{p(x)}.\]
In what follows, we assume \(t > 0\) and express \(tx\) in the form \(tx = k+u\)
where \(k\) is a non-negative integer and \(u \in [0,1)\). Then \(x =
\frac{k+u}{t}\) and:
\begin{multline*}
p'(t) = \sup_{k,u} \frac{\dist(k+u,\ZZZ)}{p(\frac{k+u}{t})} = \sup_{k,u}
\frac{\dist(u,\ZZZ)}{p(\frac{k+u}{t})} \stackrel{(*)}{=} \sup_{u \in (0,1)}
\frac{\dist(u,\ZZZ)}{p(\frac{u}{t})} \\= \sup_{u \in (0,1)}
\frac{\min(u,1-u)}{p(\frac{u}{t})} = \sup_{s \in (0,1/2]}
\max\left(\frac{s}{p(\frac{s}{t})},\frac{s}{p(\frac{1-s}{t})}\right) =
\sup_{s \in (0,1/2]} \frac{s}{p(\frac{s}{t})},
\end{multline*}
where the last equation and the one denoted by \((*)\) follow from the property
that \(p\) is monotone increasing on \(\RRR_+\). This last property combined
with \eqref{eqn:aux6} yields that \(p'\restriction{\RRR_+}\) is monotone
increasing as well. Moreover,
\begin{equation}\label{eqn:aux7}
p'(t)/t = \sup_{v \in (0,1/(2t)]} \frac{v}{p(v)}
\end{equation}
and therefore \(p'\restriction{\RRR_+}\) is quasi-concave.\par
Taking into account that \(p'\) satisfies all the conditions assumed about
\(p\), we conclude that \(p''\), or \(\reg{p}\), is quasi-concave. So, if \(p\)
is regular, it is quasi-concave. Conversely, if \(p\) is quasi-concave, then
\eqref{eqn:aux7} reduces to \(p'(t)/t = \frac{1}{2t p(1/(2t))}\). Replacing in
this argument \(p\) by \(p'\), we get that \(p'' = p\) and thus \(p\) is
reflexive, which finishes the proof.
\end{proof}

\section{Reflexive metric groups}\label{sec:refl-gen}

In this section we establish fundamental propoerties of reflexive metric groups.
Our main goal is to show that the topology of the \DUAL{} group of such a group
does not depend on the choice of a compatible reflexive \NORM. To be more
precise, first note that if a metrisable Abelian group \(G\) is equipped with
a compatible \NORM{} \(p\), the group \(G'\) as a set with a binary action is
independent of \(p\) (indeed, \(G' = \HOM{G}{\TTT}\)), but its topology (induced
by \(p'\)) does depend on \(p\), even if \(p\) is regular --- as shown by
a combination of \EXM{discrete} and \THM{comp-dual}. A situation dramatically
changes when one deals with reflexive \NORM[s]: if \(p\) and \(q\) are two
reflexive compatible \NORM[s] on \(G\), then \(p'\) and \(q'\) are equivalent.
In particular, if a compatible \NORM{} on a metrisable LCA group is reflexive,
then it is proper and hence the group is Polish (see \COR{refl-LCA} below).\par
To formulate the main result of this section, we recall the concept of
(topological) duality for topological Abelian groups \cite{cdm}. If \(G\) is
such a group, then its \emph{topological dual} is the group \(\widehat{G} =
\HOM{G}{\TTT}\) endowed with the compact-open topology. (So, \(\widehat{G}\) and
\(G'\) coincide as sets with binary actions, but their topologies may differ.)
In a similar manner one defines the \emph{topological second dual} \SECD{G} (as
the topological dual of \(\widehat{G}\)). We call \(G\) \emph{topologically
reflexive} if the homomorphism \(G \ni x \mapsto \eE_x \in \SECD{G}\) (where
\(\eE_x\) is given by \eqref{eqn:bidual}) is an isomorphism of topological
groups. (Classically a group with such a property is called briefly reflexive.
We have added here adverb ``topologically'' to distinguish this notion from
metric reflexivity.)

\begin{thm}{main-refl}
Let \((G,+,p)\) be a reflexive metric group. Then:
\begin{enumerate}[\upshape(A)]
\item the topology of \((G',\cdot,p')\) coincides with the topology of uniform
 convergence on weakly compact subsets of \(G\);
\item the topological weights of \(G\) and \(G'\) coincide;
\item \(G\) is topologically reflexive.
\end{enumerate}
In particular, if \(G\) is Polish, so is \(G'\).
\end{thm}

A proof of the above result needs some preparations. Regarding part (C), it is
worth noting that the \underline{topological} dual group of a topologically
reflexive metrisable Abelian group is in general non-metrisable. Moreover, Smith
\cite{smi} proved that all Banach spaces are topologically reflexive, in
contrast to metric duality which in this context coincides with Banach space
duality. So, in general, the reverse implication in item (C) (that is,
statement that \textit{the \NORM{} of a topologically reflexive metric Abelian
group is reflexive}) is false.\par
In the sequel we shall need the following result due to Pettis (cf.
Corollary~2.1 in \cite{pet}).

\begin{lem}{pet}
Let \(u_n\dd M \to H\ (n > 0)\) be a sequence of continuous homomorphisms from
a completely metrisable metric group \((M,\cdot,\rho)\) into a metric group
\((H,\cdot,q)\) such that the sequence \((u_n(x))_{n=1}^{\infty}\) converges in
\(H\) for any \(x \in M\). Then all \(u_n\) are equicontinuous; that is, for any
\(\epsi > 0\) there is \(\delta > 0\) such that for any \(g \in G\):
\[\rho(g) \leq \delta \implies q(u_n(g)) \leq \epsi \UP{ for any } n > 0.\]
\end{lem}

As we will see in the proof, the following result is a consequence of \LEM{pet}.
It is likely this result is known.

\begin{pro}{equi}
Let \((G,\cdot,\rho)\) be a completely metrisable metric group and
\((H,\cdot,q)\) be any metric group. If a set \(L \subset \HOM{G}{H}\) is
compact in the pointwise convergence topology, then it is equicontinuous; that
is, for any \(\epsi > 0\) there is \(\delta > 0\) such that for any \(x \in G\):
\[\rho(x) \leq \delta \implies q(f(x)) \leq \epsi \UP{ for any } f \in L.\]
\end{pro}
\begin{proof}
Assume, on the contrary, there exists a constant \(\epsi > 0\) and two sequences
\(x_1,x_2,\ldots \in G\) and \(u_1,u_2,\ldots \in L\) such that
\begin{equation}\label{eqn:aux8}
\lim_{n\to\infty} \rho(x_n) = 0 \qquad \UP{and} \qquad q(u_n(x_n)) \geq \epsi\
(n > 0).
\end{equation}
For each \(n > 0\) the set \(\{f(x_n)\dd\ f \in L\}\) is compact in \(H\). So,
after passing to a (usual) subsequence, we may and do assume that
\begin{itemize}
\item[(\(\ddag\))] the sequence \((u_n(x_k))_{n=1}^{\infty}\) converges in \(H\)
 for all \(k > 0\).
\end{itemize}
Denote by \(M\) the closure in \(G\) of the subgroup generated by all \(x_k\).
Since the set \(\{f\restriction{M}\dd\ f \in L\}\) is compact in the pointwise
convergence topology of \HOM{M}{H}, we may also (and do) assume that \(G = M\).
But then the compactness of \(L\) combined with (\ddag) implies that the closure
\(K\) of \(U \df \{u_n\dd\ n > 0\}\) in the topology of \(L\) contains at most
one element from outside the set \(U\). So, \(K\) is a countable compact
Hausdorff space and therefore it is metrisable. Hence, passing to another
subsequence, we may assume that the homomorphisms \(u_1,u_2,\ldots\) converge
pointwise. But then \LEM{pet} applies and its assertion contradicts
\eqref{eqn:aux8}.
\end{proof}

As an immediate consequence, we obtain the following result that is of great
importance for us.

\begin{thm}{comp-bdd}
If \(p\) is a complete \NORM{} on an Abelian group \(G\), then each weak*
compact subset of \(G'\) is \(p'\)-bounded. In particular, all weak* compact
sets in \(G''\) are \(p''\)-bounded.
\end{thm}
\begin{proof}
It is sufficient to show only the first statement. If \(K \subset G'\) is
compact in the weak* topology, then it is compact in the pointwise convergence
topology of \HOM{G}{\TTT}, and hence, by \PRO{equi}, there is \(\delta > 0\)
such that if \(p(x) \leq \delta\), then \(\lambda(\chi(x)) < 1/4\) for all
\(\chi \in K\). But then \(p'(\chi) \leq \frac{16M}{\delta}\) where \(M\) is
chosen so that \eqref{eqn:key} holds for \(q = \lambda\), which follows from
the proof of \LEM{key}.
\end{proof}

\begin{cor}{reflex-bdd}
Let \((G,+,p)\) be a reflexive metric group.
\begin{enumerate}[\upshape(A)]
\item Each weakly compact subset of \(G\) is \(p\)-bounded.
\item The topology of \((G',\cdot,p')\) is the topology of uniform convergence
 on weakly compact subsets of \(G\).
\end{enumerate}
\end{cor}
\begin{proof}
To show (A), repeat the argument used to explain \COR{reflex-WK}. Part (B)
follows from (A) and \THM{dual-topo}.
\end{proof}

Another consequence of \THM{comp-bdd} is formulated below.

\begin{pro}{metr-refl-topo}
If \((G,+,p)\) is a reflexive metric group, then it is topologically reflexive.
\end{pro}
\begin{proof}
It follows from \COR{reflex-bdd} that the topology of \(\widehat{G}\) is weaker
than the topology of \(G'\) and from \COR{reflex-WK} that \(p\) is complete.
Thus, \(\HOM{\widehat{G}}{\TTT} \subset \HOM{G'}{\TTT}\). This implies that
\SECD{G} and \(G''\) coincide as sets with binary actions (because of the metric
reflexivity of \(G\)). It is also clear that the function \(G \ni x \mapsto
\eE_x \in \SECD{G}\) is a continuous bijection. To convince ourselves that this
map is actually a homeomorphism, it is sufficient to show that the topologies of
\SECD{G} and of \(G''\) coincide. Since \(p''\) induces the topology of uniform
convergence on weakly compact subsets \(K\) of \(G'\) (by \COR{reflex-bdd}), we
only need to check that each such a set \(K\) is a compact subset of
\(\widehat{G}\) (that is, in the compact-open topology). But this follows from
the Ascoli-type theorem, since:
\begin{itemize}
\item the weak and the weak* topologies of \(G'\) coincide (by the metric
 reflexivity of \(G\));
\item each weak* compact set in \(G'\) is \(p'\)-bounded (thanks to
 \THM{comp-bdd}) and closed in the compact-open topology;
\item all characters from a \(p'\)-bounded subset of \(G'\) have a common
 Lipschitz constant.
\end{itemize}
\end{proof}

We now give

\begin{proof}[Proof of \THM{main-refl}]
Parts (A) and (C) are covered by, respectively, \COR{reflex-bdd} and
\PRO{metr-refl-topo}. We pass to item (B), which is the hardest part of
the result. In what follows, we will denote by \(w(X)\) the topological weight
of a topological space \(X\) and by \(C(X,\CCC)\) the space of all continuous
complex-valued functions on \(X\) equipped with the uniform convergence
topology (we will use this space only for compact \(X\)).\par
The case when \(G\) is finite is covered e.g. by the Pontryagin's duality. Hence
we assume \(G\) is infinite. Since \(G\) is isometric to \(G''\), it is
sufficient to show that
\begin{equation}\label{eqn:aux21}
w(G') \leq w(G).
\end{equation}
To this end, for each positive integer \(n\) equip \(K_n \df \bar{B}_p(n)\) with
the topology induced from the weak one of \(G\), and recall that:
\begin{itemize}
\item \(K_n\) is a compact Hausdorff space (thanks to \COR{reflex-WK});
\item \(w(C(K_n,\CCC)) = w(K_n)\) (by the compactness of \(K_n\));
\item \(G'\) embeds (as a topological space) into \(\prod_{n=1}^{\infty}
 C(K_n,\CCC)\) via a map
 \[\chi \mapsto (\chi\restriction{K_n})_{n=1}^{\infty}\]
 (thanks to part (B) of \THM{dual-topo}).
\end{itemize}
So, to establish \eqref{eqn:aux21} we only need to check that
\begin{equation}\label{eqn:aux22}
w(K_n) \leq w(G)
\end{equation}
for each \(n\). To simplify notation, we fix \(n\) and set \(Q = K_n\). Further,
we equip \(M_m \df \bar{B}_{p'}(m) \subset G'\) with the topology inherited from
the weak* topology of \(G'\). We infer from \THM{dual-topo} that \(M_m\) is
a compact Hausdorff space. Observe that the function
\[\Phi_m\dd Q \ni x \mapsto \eE_x\restriction{M_m} \in C(M_m,\CCC)\]
is continuous if \(C(M_m,\CCC)\) is equipped with the pointwise convergence
topology. In particular, \(E_m \df \Phi_m(Q)\) is compact in that topology. As
it is also a bounded set in \(C(M_m,\CCC)\), it follows from a result due to
Grothendieck \cite{gro} that \(\Phi_m(Q)\) is compact in the weak topology of
the Banach space \(C(M_m,\CCC)\) (here the \textit{weak topology} is understood
as it is in the Banach space theory). In other words, \(E_m\) is an Eberlein
compactum. It is well-known that all such compacta \(E\) satisfy \(\OPN{dens}(E)
= w(E)\) where \(\OPN{dens}(E)\) is the least cardinality among the sizes of
dense sets in \(E\). Since the map
\[Q \ni x \mapsto (\Phi_m(x))_{m=1}^{\infty} \in \prod_{m=1}^{\infty} E_m\]
is one-to-one (so, it is a topological embedding) and countable products of
Eberlein compacta are Eberlein as well, we conclude that \(Q\) is an Eberlein
compactum. Consequently, \(w(Q) = \OPN{dens}(Q)\). Now it is sufficient to take
a set \(D \subset Q\) which is dense in \(Q\) in the topology of \(G\) induced
by the \NORM{} \(p\) and satisfies \(\OPN{card}(D) \leq w(G)\) to conclude that
\(D\) is dense in \(Q\) in the weak topology and therefore \(w(Q) \leq
\OPN{card}(D)\), which finishes the proof of \eqref{eqn:aux22} and of
the theorem.
\end{proof}

\begin{cor}{refl-LCA}
If \(p\) is a non-proper compatible \NORM{} on a metrisable LCA group \(G\),
then \(p\) is not semireflexive.
\end{cor}
\begin{proof}
Assume \(p\) is semireflexive. Then \(\reg{p}\) is compatible and reflexive. So,
\(G\) is weakly \(\sigma\)-compact, thanks to \COR{reflex-WK}, and actually
\(G\) is \(\sigma\)-compact, by \THM{glick}. Consequently, \(G\) is Polish, and
admits a reflexive proper compatible \NORM{} \(q\) (see \THM{comp-dual}). Now it
follows from \THM{main-refl} that \(p'\) and \(q'\) are equivalent and thus, by
\THM{comp-dual}, \(p\) is proper.
\end{proof}

\begin{cor}{bdd-refl}
If \(p\) is a bounded \NORM{} on an Abelian group \((G,+)\), then either
\((G,+,p)\) is compact or \(p\) is not semireflexive.
\end{cor}
\begin{proof}
Since all compatible \NORM[s] on metrisable compact Abelian groups are
semi-reflexive, it suffices to show the compactness of \(G''\) for
semi-reflexive \(p\) (as in that case \(G\) is topologically isomorphic to
\(G''\)). It readily follows from the boundedness of \(p\) that \(p'\) induces
the discrete topology on \(G'\). In particular, \(p'\) is a reflexive compatible
\NORM{} on a LCA group. So, \COR{refl-LCA} implies that \(p'\) is proper and
consequently \(G''\) is compact (since \(G'\) is discrete; cf. \THM{comp-dual}).
\end{proof}

\begin{rem}{topo-metr-refl}
Let us call a metrisable Abelian group \emph{reflexifiable} if it admits
a compatible \NORM{} that is reflexive. It seems to be an interesting question
of when a metrisable Abelian group is reflexifiable. \THM{main-refl} provides us
some insight into this issue. Firstly, each such a group \(G\) needs to be
complete and topologically reflexive. Secondly, if we endow \(G\) and \(G^* \df
\HOM{G}{\TTT}\) with the weak and weak* topologies (respectively), both these
topological groups need to be \emph{hemicompact} (a topological Hausdorff space
\(X\) is hemicompact if it may be covered by a countable family of its compact
subsets such that any compact set in \(X\) is contained in some member of that
family.) Moreover, the given topology of \(G\) needs to coincide with
the topology of uniform convergence on weak* compact subsets of \(G^*\).
Finally, denoting by \(G'\) the group \(G^*\) endowed with topology of uniform
convergence on weakly compact subsets of \(G\), each continuous character on
\(G'\) needs to have the form \(\eE_x\) for some \(x \in G\). This last property
seems to be the hardest part of the above list. (Note that \(G'\) is metrisable
if \(G\) is weakly hemicompact.)\par
As it is easily seen, the product of a finite number of reflexifiable
topological Abelian groups is reflexifiable as well.
\end{rem}

\begin{exm}{reflexif}
Reflexive Banach spaces, when considered as topological Abelian groups, can be
characterised as reflexifiable groups that are both connected and locally
connected and whose fundamental groups are trivial. Instead of giving a detailed
proof, here we list its main steps. Below \(G\) denotes a metrisable Abelian
group with all the properties specified above, \(p\) is a reflexive compatible
\NORM{} on \(G\) and \((\RRR,+)\) is endowed with the natural \NORM.
\begin{enumerate}[(ST1)]
\item Being completety metrisable, \(G\) is both arcwise and locally arcwise
 connected.
\item Each character on \(G\) is of the form \(q \circ \phi\) where \(\phi \in
 G^* \df \HOM{G}{\RRR}\) and \(q\dd \RRR \ni t \mapsto e^{2\pi it} \in \TTT\).
 The assignment \(G' \ni \phi \mapsto q \circ \phi \in G^*\) is a bijective
 homomorphism of abstract groups. Denote the inverse assignment by \(\chi
 \mapsto \chi_{\RRR}\). So, for \(\chi \in G'\), \(\chi_{\RRR} \in G^*\)
 satisfies \(q \circ \chi_{\RRR} = \chi\). Endow \(G^*\) with a topology such
 that the mapping \(G' \ni \chi \mapsto \chi_{\RRR} \in G^*\) is
 a homeomorphism. In this way \(G^*\) is a completely metrisable Abelian group.
 Note also that \(G_{\RRR}\) is a real vector space (with pointwise actions).
\item Each \(\phi \in G^*\) is bounded on \(\bar{B}_p(R)\) for all \(R > 0\)
 (by the weak compactness of closed \(p\)-balls and \THM{glick}).
\item If \(\phi_1,\phi_2,\ldots \in G^*\) converge to \(0\) uniformly on
 \(p\)-bounded subsets of \(G\), then these homomorphisms converge to \(0\) in
 the topology of \(G^*\).
\item For any \(\phi \in G^*\), the mapping \(\RRR \ni t \mapsto t \phi \in
 G^*\) is continuous.
\item The homomorphism \(G^* \ni \phi \mapsto \frac12 \phi \in G^*\) is
 continous.
\item \(G^*\) is a topological vector space. (This follows only from (ST5),
 (ST6) and complete metrisability of \(G^*\). It is likely that such a general
 theorem on topological Abelian groups is known.)
\item The topology of \(G^*\) coincides with the topology of uniform convergence
 on \(p\)-bounded subsets of \(G\).
\item \(G^*\) is a completely metrisable locally convex space.
\item \(G\) is a completely metrisable locally convex space as well. (Since we
 can apply all the previous steps to \(G'\) which is isomorphic to \(G^*\) to
 conclude analogous properties of \(G''\) which is isomorphic to \(G\).)
\item All closed \(p\)-balls in \(G\) are compact in the weak topology of
 a locally convex space.
\item Considered as a locally convex space, \(G\) is normable and a reflexive
 Banach space. (Since \(G\) contains a weakly compact neighbourhood of \(0\).)
\end{enumerate}
\end{exm}

\begin{rem}{functoreflex}
In \REM{functor} we have given an example of a continuous homomorphism from
a complete metric Abelian group into a reflexive metric Abelian group whose dual
homomorphism is discontinuous. If we assume, in addition, that the source group
(that is, the domain) is reflexive as well, the aforementioned phenomenon does
not occur, as all continuous homomorphisms between reflexive groups are
automatically bounded, which easily follows from the weak compactness of closed
balls (in such groups) and from the boundedness of weakly compact sets in such
groups.
\end{rem}

\section{Duality and regularisation for quasi-\NORM[s]}\label{sec:quasi}

In this section we propose a more general (and more abstract) approach to
duality in Abelian groups, where no topological structures are considered
(in particular, all groups are endowed with no topologies). These ideas will be
applied in the next section to a non-metrisable groups.\par
From now on to the end of this section, all groups are Abelian and endowed with
no topologies, and we use multiplicative notation for their actions. For any
such a group \(G\), \(G^{\#}\) stands for the group of all homomorphisms from
\(G\) into \(\TTT\). Additionally, we set \(G^{\#\#} \df (G^{\#})^{\#}\) and
introduce \emph{canonical embedding} \(J_G\dd G \to G^{\#\#}\) given by
a standard rule: \(J_G(x) = \eE_x\) where \(\eE_x(\chi) = \chi(x)\) (where \(x
\in G\) and \(\chi \in G^{\#}\)). (Recall that \(J_G\) is always one-to-one and
it is surjective iff \(G\) is finite.) Further, for any set \(A \in G\) we use
\(A^{\perp}\) to denote the set of all \(\phi \in G^{\#}\) such that
\(\phi\restriction{A} \equiv 1\).

\begin{dfn}{quasinorm}
A \emph{quasi-\NORM} on a group \(G\) is any function \(p\dd G \to [0,\infty]\)
such that:
\begin{itemize}
\item \(p(e_G) = 0\);
\item \(p(x^{-1}) = p(x)\) for all \(x \in G\);
\item \(p(xy) \leq p(x)+p(y)\) for any \(x, y \in G\).
\end{itemize}
For any quasi-\NORM{} \(p\) on \(G\) we define:
\begin{itemize}
\item its \emph{dual} quasi-\NORM{} \(p^{\#}\) on \(G^{\#}\) by:
 \[p^{\#}(\chi) = \sup \left\{\frac{\lambda(\chi(g))}{p(g)}\dd\ g \in G,\
 \chi(g) \neq 1\right\} \qquad (\chi \in G^{\#}),\]
 under the following conventions: \(\sup(\varempty) \df 0\), \(\frac{a}{0} \df
 \infty\) and \(\frac{a}{\infty} \df 0\) for all \(a \in (0,\infty)\);
\item \(p^{\#\#} \df (p^{\#})^{\#}\) (a quasi-\NORM{} on \(G^{\#\#}\));
\item \emph{regularisation} \(\reg{p}\) of \(p\) (a quasi-\NORM{} on \(G\)) by:
 \(\reg{p}(x) = p^{\#\#}(\eE_x)\ (x \in G)\).
\end{itemize}
The quasi-\NORM{} \(p\) is said to be \emph{regular} if \(p = \reg{p}\). We use
\(\REG(G)\) to denote the set of all regular quasi-\NORM[s] on \(G\).
Additionally, \(\ker(p)\) and \(\fin(p)\) will stand for the sets, respectively,
\(p^{-1}(\{0\})\) and \(p^{-1}(\RRR)\). Note that \(\ker(p)\) and \(\fin(p)\)
are subgroups of \(G\). Further, for each real \(r > 0\) we use \(B_p(r)\) and
\(\bar{B}_p(r)\) to denote the sets of all \(g \in G\) such that \(p(g) < r\)
and \(p(g) \leq r\), respectively. 
\end{dfn}

Basic properties of the notions introduced above are collected below.

\begin{pro}{basic}
For any quasi-\NORM[s] \(p\) and \(q\) on \(G\):
\begin{enumerate}[\upshape(a)]
\item if \(p \leq q\), then \(\reg{p} \leq \reg{q}\) and \(p^{\#} \geq q^{\#}\);
\item if \(p, q \in \REG(G)\), then \(p \leq q\) iff \(p^{\#} \geq q^{\#}\);
\item if \(p \in \REG(G)\) and \(p^{\#} = q^{\#}\), then \(p \leq q\); in
 particular, \(\reg{p} \leq p\);
\item \(\reg{p} \in \REG(G)\) and \((\reg{p})^{\#} = p^{\#}\);
\item \(p^{\#} \in \REG(G^{\#})\);
\item for any quasi-\NORM{} \(\tau\) on \(G^{\#}\), the quasi-\NORM{}
 \(\sigma\dd G \to [0,\infty]\) given by \(\sigma(x) = \tau^{\#}(\eE_x)\) is
 regular;
\item \(\ker(p^{\#}) = \fin(p)^{\perp}\) and \(\fin(p^{\#})\) consists of all
 homomorphisms \(\phi \in G^{\#}\) that are continuous w.r.t. \(p\); that is,
 \(\phi \in \fin(p^{\#})\) iff \(\lim_{n\to\infty} \phi(g_n) = 1\) for any
 sequence \((g_n)_{n=1}^{\infty} \in G\) such that \(\lim_{n\to\infty} p(g_n) =
 0\) \UP(in particular, \(\ker(p) \subset \ker(\phi)\) for all \(\phi \in
 \fin(p^{\#})\)\UP).
\end{enumerate}
\end{pro}
\begin{proof}
First of all, notice that for all \(x \neq e_G\):
\begin{equation}\label{eqn:qv-reg}
\reg{p}(x) = \sup_{\chi \in G^{\#} \setminus \{x\}^{\perp}}
\left(\inf_{g \notin \ker(\chi)}
\frac{\lambda(\chi(x))}{\lambda(\chi(g))} p(g)\right)
\end{equation}
(cf. \PRO{regular}), which easily implies part (a) and that \(\reg{p} \leq p\).
Note also that (b) follows from (a), since regular \(p\) is determined by
\(p^{\#\#}\), and that the former part of (c) is implied by its latter part and
by (b).\par
Further, observe that for any \(\epsi > 0\),
\[\lambda(\chi(g)) \leq (p(g)+\epsi)(p^{\#}(\chi)+\epsi) \qquad (g \in G,\ \chi
\in G^{\#})\]
(with natural conventions that \(\infty+a = \infty \cdot a = a \cdot \infty =
\infty \cdot \infty = \infty\) for any \(a \in (0,\infty)\)) and
\(p^{\#}(\chi)\) is the least possible value from \([0,\infty]\) for which
the above inequality holds for all \(g \in G\) (with fixed \(\chi\)). Using this
characterisation and proceeding similarly as in the proof of \PRO{regular} one
gets (e) and that \((\reg{p})^{\#} = p^{\#}\). Since the first claim of (d) is
a special case of (f), it remains to show (f) and (g). Instead of giving
a detailed proof of (f), we refer the reader to \THM{sub-reg} (stated below)
which combined with (e) yields (f). Finally, the former claim of (g) is left to
the reader, whereas the latter one follows from the proof of \LEM{key}.
\end{proof}

As a conseqence of the above result, we obtain the following two important
properties of regular quasi-\NORM[s].

\begin{thm}{sup-reg}
Let \(\{p_s\}_{s \in S}\) be an arbitrary non-empty collection of regular
quasi-\NORM[s] on a group \(G\). Then its pointwise supremum---that is,
the function \(\bigvee_{s \in S} p_s\) given by the formula
\[(\bigvee_{s \in S} p_s)(g) = \sup_{s \in S} p_s(g) \qquad (g \in G)\]
is a regular quasi-\NORM{} as well.
\end{thm}
\begin{proof}
For simplicity, put \(q \df \bigvee_{s \in S} p_s\). It is clear that \(q\) is
a quasi-\NORM. We have already known that \(\reg{q} \leq q\). On the other hand,
for any \(s \in S\) we have \(p_s = \reg{(p_s)} \leq \reg{q}\) (since \(p_s \leq
q\)) and consequently \(q \leq \reg{q}\), which finishes the proof.
\end{proof}

\begin{pro}{prod-reg}
Let \(p_s\) be a regular quasi-\NORM{} on a group \(G_s\) \UP(for \(s \in S \neq
\varempty\)\UP). Then the quasi-\NORM{} \(q\dd \prod_{s \in S} G_s \to
[0,\infty]\) given by
\[q((g_s)_{s \in S}) \df \sup_{s \in S} p_s(g_s) \qquad ((g_s)_{s \in S} \in
\prod_{s \in S} G_s)\]
is regular as well.
\end{pro}
\begin{proof}
To simplify notation, put \(H_s \df G_s^{\#}\), \(r_s \df p_s^{\#}\) and \(G \df
\prod_{s \in S} G_s\), denote by \(H = \bigoplus_{s \in S} H_s\) the direct
product of all the groups \(H_s\), and define \(\tau\dd H \to [0,\infty]\) and
\(\sigma\dd G \to [0,\infty]\) by the rules: \(\tau((h_s)_{s \in S}) =
\sum_{s \in S} r_s(h_s)\) and
\[\sigma((g_s)_{s \in S}) = \tau^{\#}((J_{G_s}(g_s))_{s \in S}).\]
(Note that \(H^{\#}\) is naturally isomorphic to the full product
\(\prod_{s \in S} H_s^{\#}\) and therefore the above
\((J_{G_s}(g_s))_{s \in S}\) may naturally be considered as an element of
\(H^{\#}\).) As the mapping \(G \ni (g_s)_{s \in S} \mapsto
(J_{G_s}(g_s))_{s \in S} \in H^{\#}\) is a monomorphism, it is sufficient to
check that \(q = \sigma\) (thanks to part (e) of \PRO{basic} and to
\THM{sub-reg} stated below). To this end, observe that
\[\tau^{\#}((\phi_s)_{s \in S}) = \sup_{s \in S} r_s^{\#}(\phi_s) \qquad
((\phi_s)_{s \in S} \in \prod_{s \in S} H_s^{\#})\]
and that \(r_s^{\#}(J_{G_s}(g_s)) = p_s^{\#\#}(J_{G_s}(g_s)) = p_s(g_s)\) for
any \(s \in S\) and \(g_s \in G_s\), since \(p_s\) is regular. These two
observations show that indeed \(\sigma = q\) and finish the proof.
\end{proof}

\begin{rem}{inf-reg}
\THM{sup-reg} shows that for any group \(G\) the set \(\REG(G)\) is order
complete. Since the pointwise minimum of two quasi-\NORM[s] is not
a quasi-\NORM{} in general, a natural question arises of how to find order
infima of non-empty sets of regular quasi-\NORM[s]. With the aid of the concept
of dual quasi-\NORM[s] this issue may be solved in a straightforward way as
follows: for any collection \(\{p_s\}_{s \in S}\) of regular quasi-\NORM[s] on
a group \(G\) define \(q\dd G^{\#} \to [0,\infty]\) and \(\bigwedge_{s \in S}
p_s\dd G \to [0,\infty]\) by \(q(\phi) = \sup_{s \in S} p_s^{\#}(\phi)\) and
\((\bigwedge_{s \in S} p_s)(g) = q^{\#}(J_G(g))\). It follows from part (f) of
\PRO{basic} that \(\bigwedge_{s \in S} p_s\) is a regular quasi-\NORM{} on
\(G\) and from the other parts of that result that it is the infimum of the set
\(\{p_s\}_{s \in S}\) in the poset \(\REG(G)\).\par
It is also worth noting that, in general, regular quasi-\NORM[s] on \(G^{\#}\)
may not be of the form \(p^{\#}\) where \(p\) is a quasi-\NORM{} on \(G\).
However, for any finite group \(G\) the assignment \(p \mapsto p^{\#}\)
establishes a one-to-one correspondence between regular quasi-\NORM[s] \(p\) on
\(G\) and regular ones \(q = p^{\#}\) on \(G^{\#}\).
\end{rem}

\begin{exm}{trivial}
As it may easily be shown, the quasi-\NORM[s] \(\zeroqv{G}, \inftyqv{G}\dd G
\to [0,\infty]\) defined as \(\zeroqv{G}(g) = 0\) and \(\inftyqv{G}(g) =
\infty\) for each \(g \in G \setminus \{e_G\}\) are, respectively, the least and
the greatest elements of the poset \(\REG(G)\). It may also easily be verified
that \(\zeroqv{G}^{\#} = \inftyqv{G^{\#}}\) and \(\inftyqv{G}^{\#} =
\zeroqv{G^{\#}}\).
\end{exm}

Less obvious property of regular quasi-\NORM[s] is formulated in the following

\begin{thm}{sub-reg}
A restriction of a regular quasi-\NORM{} to a subgroup is regular as well. More
precisely, if \(p \in \REG(G)\) and \(H\) is a subgroup of \(G\), then
\(p\restriction{H} \in \REG(H)\).
\end{thm}
\begin{proof}
To simplify notation, set \(q \df \reg{(p\restriction{H})}\). We need to show
that \(q(h) = p(h)\) for any \(h \in H\). We have already known that \(q \leq
p\restriction{H}\). To show the reverse inequality, fix \(u \in H \setminus
\{e_G\}\) and real \(m < p(u)\), and using \eqref{eqn:qv-reg} and the regularity
of \(p\) choose \(\rho \in G^{\#}\) such that \(\rho(u) \neq 1\) and
\(\lambda(\rho(u)) p(g) \geq m \lambda(\rho(g))\) for all \(g \in G\). Then
\(\chi \df \rho\restriction{H}\) belongs to \(H^{\#}\) and satisfies
\[\inf_{h \in H \setminus \ker(\chi)} \frac{\lambda(\chi(u))}{\lambda(\chi(h))}
p(h) \geq m,\] and thus \(q(u) \geq m\) (again by \eqref{eqn:qv-reg}), which
finishes the proof.
\end{proof}

The following result will find applications in the next section.

\begin{cor}{qv-reg}
Let \(p\) be a quasi-\NORM{} on a group \(G\) and \(H\) be an arbitrary subgroup
of \(G^{\#}\) that contains all homomorphisms from \(G\) into \(\TTT\) which are
continuous w.r.t. \(p\). Then, \(q \df p^{\#}\restriction{H}\) is regular and
\(\reg{p}(g) = q^{\#}(\eE_g\restriction{H})\) for all \(g \in G\).
\end{cor}
\begin{proof}
The former claim is a special case of the previous theorem (thanks to part (e)
of \PRO{basic}). To see the latter, observe that, for \(g \neq e_G\):
\begin{align*}
\reg{p}(g) &= \sup \left\{\frac{\lambda(\phi(g))}{p^{\#}(\phi)}\dd\ \phi \in
G^{\#},\ \phi(g) \neq 1\right\}\\ &=
\sup \left\{\frac{\lambda(\phi(g))}{p^{\#}(\phi)}\dd\ \phi \in G^{\#},\ \phi(g)
\neq 1,\ p^{\#}(\phi) < \infty\right\}\\ &\stackrel{(*)}{=}
\sup \left\{\frac{\lambda(\phi(g))}{q(\phi)}\dd\ \phi \in H,\ \phi(g) \neq
1\right\} = q^{\#}(\eE_g\restriction{H}),
\end{align*}
where the equation denoted by \((*)\) follows from the assumption about \(H\).
\end{proof}

\section{Metric duality for non-metrisable Abelian groups}\label{sec:metr-str}

In this part we apply the results from the previous section to discuss metric
duality for non-metrisable topological Abelian groups. We start from recalling
basic concepts of topologising Abelian groups by families of quasi-\NORM[s].\par
Let \(G\) be an Abelian group. A collection of quasi-\NORM[s] \(\PpP =
\{p_s\}_{s \in S}\) on an Abelian group \(G\) is said to be \emph{separating} if
for any \(g \in G \setminus \{e_G\}\) there is an index \(s \in S\) such that
\(p_s(g) > 0\). If \(\PpP\) is such a family (that is, if \(\PpP\) is
separating), it induces a topology \(\tau_{\PpP}\) on \(G\) which makes \(G\)
a Hausdorff topological group: finite intersections of any translations of sets
\(B_{p_s}(r)\) (where \(s \in S\) and real \(r > 0\) are arbitrary) form a basis
of \(\tau_{\PpP}\). We call this topology \emph{induced by \(\PpP\)}. Observe
that \(\PpP\) consists of continuous quasi-\NORM[s] (w.r.t. \(\tau_{\PpP}\)). In
particular, the groups \(\fin(p_s)\) are all open in \(G\).\par
Conversely, for any (Hausdorff) topological Abelian group \(A\) there is
a separating collection \(\QqQ\) of (continuous) real-valued quasi-\NORM[s] on
\(A\) that induces the given topology of \(A\). (This statement is an immediate
consequence of a classical result due to A.A.~Markov on the existence of
continuous semi\NORM[s] with arbitrarily small balls of radius 1; consult e.g.
\cite[Theorem~3.3.9]{a-t}.)\par
As it is commonly practised, in what follows, we will assume that a collection
\(\PpP\) of quasi-\NORM[s] is \emph{upward directed}; that is, we will assume
that: \[\forall p,q \in \PpP\ \exists r \in \PpP\dd\ p \leq r,\ q \leq r.\]
(Similarly, we will say that \(\PpP\) is \emph{downward directed} if for any
two its members \(p\) and \(q\) there is \(v \in \PpP\) such that \(v \leq p\)
and \(v \leq q\).)\par
As it was done in \DEF{dual-value}, for a topological Abelian group \(G\) we
set \(G' \df \HOM{G}{\TTT}\) and call \(G'\) the \emph{dual} group. (At this
very preliminary stage of our considerations, the group \(G'\) is equipped with
no topology.) For any \textbf{continuous} quasi-\NORM{} \(p\) on \(G\), we use
\(p'\) to denote the quasi-\NORM{} \(p^{\#}\restriction{G'}\) on \(G'\) and call
it \emph{dual} to \(p\). Additionally, \(\reg{p}\) stands for the regularisation
of \(p\) defined in \DEF{quasinorm}, and \(\kappa_G\dd G \to (G')^{\#}\) is
a homomorphism given by the rule: \(\kappa_G(g) \df
\eE_G(g)\restriction{G'}\).\par
As an immediate consequence of \COR{qv-reg}, we obtain

\begin{pro}{qv-dual-top}
For any continuous quasi-\NORM{} \(p\) on \(G\), \(p'\) and \(\reg{p}\) are
regular quasi-\NORM[s]. Moreover, \(\reg{p}(g) = (p')^{\#}(\kappa_G(g))\).
\end{pro}
\begin{proof}
It suffices to observe that each homomorphism taken from \(G^{\#}\) that is
continuous w.r.t. \(p\) belongs to \(G'\) and to apply \COR{qv-reg}.
\end{proof}

To define a metric structure on a non-metrisable (Abelian) group, we make
a crucial observation:

\begin{pro}{metr-str}
For an upward directed separating collection \(\PpP\) of \textbf{regular}
quasi-\NORM[s] on a topological Abelian group \((G,\cdot)\) that induces
the topology of \(G\) \tfcae
\begin{enumerate}[\upshape(i)]
\item the collection \(\PpP' \df \{p'\dd\ p \in \PpP\}\) is both separating
 \UP(on \(G'\)\UP) and upward directed;
\item \(\PpP\) is downward directed and
 \begin{equation}\label{eqn:fin}
 \forall g \in G\ \exists p \in \PpP\dd\ p(g) < \infty;
 \end{equation}
\item \(\PpP'\) is separating, both downward and upward directed and for any
 \(\phi \in G'\) there is \(q \in \PpP'\) such that \(q(\phi) < \infty\).
\end{enumerate}
\end{pro}
\begin{proof}
Of course, (i) follows from (iii). To see that (ii) is implied by (i), for \(p,
q \in \PpP\), using (i), take \(v \in \PpP\) such that \(p' \leq v'\) and \(q'
\leq v'\) and then conclude that \((v')^{\#} \leq (p')^{\#}\) as well as
\((v')^{\#} \leq (q')^{\#}\), from which it follows (thanks to \PRO{qv-dual-top}
and the regularity of \(p, q\) and \(v\)) that \(v \leq p\) and \(v \leq q\).
Further, since each homomorphism from \(G\) into \(\TTT\) that vanishes on
the subgroup \(H\) of \(G\) generated by \(\bigcup_{p \in \PpP} \fin(p)\) is
continuous, and \(\PpP'\) is separating, we infer from part (g) of \PRO{basic}
that \(H = G\). So, for any \(g \in G\) there are \(h_1 \in \fin(p_1),\ldots,h_n
\in \fin(p_n)\) (with suitably chosen \(p_1,\ldots,p_n\) from \(\PpP\)) such
that \(g = h_1 \cdot \ldots \cdot h_n\). Finally, as we have already shown that
\(\PpP\) is downward directed, there exists \(q \in \PpP\) that is pointwise
upper bounded by each of \(p_j\) (\(j=1,\ldots,n\)). So, \(q(g) \leq
\sum_{j=1}^n p_j(h_j) < \infty\), which finishes the proof of (ii).\par
It remains to show that (ii) is followed by (iii). We conclude from part (a) of
\PRO{basic} that \(\PpP'\) is both downward and upward directed (as \(\PpP\) is
so). Further, if \(\phi \in G'\) differs from \(e_{G'}\), we take arbitrary \(g
\in G\) such that \(\phi(g) \neq 1\) and then, using \eqref{eqn:fin}, find \(p
\in \PpP\) with \(p(g) < \infty\). Consequently, \(p'(\phi) \geq
\frac{\lambda(\phi(g))}{p(g)}\) and hence \(\PpP'\) is separating. Finally,
since \(\phi\) is continuous and the family \(\PpP\) is upward directed and
induces the topology of \(G\), there are \(q \in \PpP\) and \(\epsi > 0\) such
that \(\phi(B_q(\epsi)) \subset B_{\lambda}(\frac14)\). But then \(q'(\phi) <
\infty\) (cf. the proof of \LEM{key}) and we are done.
\end{proof}

Although the equivalence expressed in the above result holds (only) for families
\(\PpP\) of regular quasi-\NORM[s], for simplicity we introduce the following

\begin{dfn}{metr-str}
A \emph{metric structure} on an Abelian group \(G\) (equipped with no topology)
is any non-empty collection \(\PpP\) of quasi-\NORM[s] on \(G\) that is
separating and both downward and upward directed, and satisfies condition
\eqref{eqn:fin}.\par
A \emph{metric structure} on a \textbf{topological} Abelian group \(A\) is
a metric structure on \(A\) that induces the topology of \(A\).\par
For any metric structure \(\PpP\) on a topological Abelian group \(G\)
the collection \(\PpP' \df \{p'\dd\ p \in \PpP\}\) is called \emph{dual to
\(\PpP\)}. \(\PpP'\) is a metric structure on \(G'\), thanks to \PRO{metr-str}.
\end{dfn}

\begin{exm}{add-zero}
It may easily be shown that any (Hausdorff) topological Abelian group \(G\)
admits a metric structure. Indeed, we can find a separating collection
\(\PpP_0\) of semi\NORM[s] on \(G\) that induces the topology of \(G\) and then
build a metric structure \(\PpP\) by adding to \(\PpP_0\) the quasi-\NORM[s]
\(\zeroqv{G}\) and \(\max(p_1,\ldots,p_n)\) where \(n > 0\) and \(p_1,\ldots,
p_n \in \PpP_0\) are arbitrary. However, when \(\zeroqv{G}\) belongs to
\(\PpP\), then \(\inftyqv{G'}\) is in \(\PpP'\) and thus \(\PpP'\) induces
the discrete topology on \(G'\).
\end{exm}

Let \(\PpP\) be a metric structure on a topological Abelian group \((G,+)\).
We equip \(G'\) with the topology induced by \(\PpP'\) to obtain a topological
group \((G',\cdot)\), which we will denote by \((G,+,\PpP)' = (G',\cdot,\PpP')\)
and call \emph{metric dual of \(G\) induced by \(\PpP\)}. Applying this
procedure to \((G',\cdot)\) in place of \((G,+)\), we obtain \emph{metric second
dual of \(G\)} (\emph{induced by \(\PpP\)}), to be denoted by \((G,+,\PpP)'' =
(G'',\cdot,\PpP'') = (G',\cdot,\PpP')'\). For simplicity, for each \(p \in
\PpP\) we will write \(p''\) to denote \((p')'\) where the latter quasi-\NORM{}
is computed with respect to the collection \(\PpP'\). It is worth underlying
here that both the topology of \(G'\) as well as the domain of \(p''\) does
depend on the metric structure \(\PpP\) of the topological group \(G\).\par
Continuing the above story, for any \(g \in G\) take \(p \in \PpP\) with \(p(g)
< \infty\) and observe that then \(\lambda(\chi(g)) \leq M p'(\chi)\) for any
\(\chi \in G'\) where \(M\) is an arbitrary real number greater than \(p(g)\)
(we enlarge \(p(g)\) to avoid indeterminate form ``\(0 \cdot \infty\)''). This
shows that \(\eE_g\restriction{G'}\) is continuous and hence we are allowed to
consider \(\kappa_G\) as a homomorphism from \(G\) into \(G''\). Moreover,
\PRO{qv-dual-top} yields that for each \(p \in \PpP\),
\[\reg{p}(g) = p''(\kappa_G(g)) \qquad (g \in G),\]
which is consistent with considerations from Section~2 (cf. \DEF{reg-refl}). In
particular, regularity of \(p\) is equivalent to the statement that \(\kappa_G\)
is \emph{isometric} w.r.t. \(p\) and \(p''\).

\begin{dfn}{metr-str-refl}
A metric structure \(\PpP\) on an Abelian group \((G,+)\) is said to be:
\begin{itemize}
\item \emph{regular} if it consists of regular quasi-\NORM[s];
\item \emph{semireflexive} if \(\kappa_G\dd (G,+,\PpP) \to (G'',\cdot,\PpP'')\)
 is an isomorphism of topological groups;
\item \emph{reflexive} if it is both regular and semireflexive.
\end{itemize}
\end{dfn}

The main goal of this section says that each LCA group admits a reflexive metric
structure such that its dual structure induces the Pontryagin's topology of
the dual group. To prove it, we need a slight generalisation of \THM{comp-dual}
which reads as follows.

\begin{lem}{qv-proper}
For a continuous quasi-\NORM{} \(p\) on a LCA group \(G\) \tfcae
\begin{enumerate}[\upshape(i)]
\item \(p'\) is continuous in the Pontryagin's topology of \(G'\);
\item \(p\) is proper; that is, the sets \(\bar{B}_p(r)\) are compact in \(G\)
 for all finite \(r > 0\).
\end{enumerate}
Moreover, if \(p\) is proper, so are \(p'\) and \(\reg{p}\).
\end{lem}
\begin{proof}
First assume \(p\) is proper. Similarly as presented in the proof of part (B) of
\THM{dual-topo}, one shows that \(p'\) induces topology of uniform convergence
on \(p\)-bounded subsets of \(G\). Consequently, this topology is weaker than
the compact-open one on \(G'\) (that is, than the Pontryagin's topology of
\(G'\)), which means that (i) holds. Conversely, if \(p'\) is continuous in
the Pontryagin's topology, then for each finite \(r > 0\) the set
\(\bar{B}_{p'}(r)\) is an equicontinuous family and a closed set in this
topology. So, by the Ascoli-type theorem, this set is compact and therefore
\(p'\) is proper. Now applying all that has been already established in this
proof to \(p'\) in place of \(p\), we get that \(p''\) is proper. Consequently,
\(\reg{p}\) and \(p\) are proper as well (as \(\bar{B}_p(r) \subset
\bar{B}_{\reg{p}}(r)\)).
\end{proof}

For simplicity, let us call a metric structure on a LCA group \emph{proper} if
all its quasi-\NORM[s] are proper (consult \LEM{qv-proper}).\par
And now it is time for the main result of the section.

\begin{thm}{LCA-metr-str}
Let \(G\) be a LCA group.
\begin{enumerate}[\upshape(A)]
\item A metric structure \(\PpP\) on \(G\) is proper iff \(\PpP'\) induces
 the Pontryagin's topology of \(G'\).
\item \(G\) admits a metric structure that is both proper and reflexive.
\item The assignment \(\PpP \leftrightarrow \PpP'\) establishes a one-to-one
 correspondence between reflexive proper metric structures \(\PpP\) on \(G\)
 and such structures \(\QqQ = \PpP'\) on \(G'\).
\end{enumerate}
\end{thm}
\begin{proof}
In what follows, we consider \(G'\) with the Pontryagin's topology.\par
We start from (A). It follows from \LEM{qv-proper} that if \(\PpP'\) consists of
continuous quasi-\NORM[s], then \(\PpP\) is proper. Conversely, if \(\PpP\) is
proper, all quasi-\NORM[s] in \(\PpP'\) are continuous and proper (by the same
lemma). This implies that the topology of \((G',\cdot,\PpP')\) is weaker than
the Pontryagin's one. So, to finish the proof of (A), we only need to show that
nets in \(G'\) that converge to the neutral element w.r.t. the topology induced
by \(\PpP'\) also converge in the compact-open topology. To this end, fix a net
\((\phi_{\sigma})_{\sigma\in\Sigma} \subset G'\) such that
\begin{equation}\label{eqn:aux31}
\lim_{\sigma\in\Sigma} p'(\phi_{\sigma}) = 0
\end{equation}
for any \(p \in \PpP\), and a quasi-\NORM{} \(q\) from \(\PpP\). It follows from
\eqref{eqn:aux31} that \(q'(\phi_{\sigma}) \leq 1\) for \(\sigma\) sufficiently
large. But then, since \(q'\) is proper, the net
\((\phi_{\sigma})_{\sigma\in\Sigma}\) has a `tail' contained in a compact set in
the compact-open topology of \(G'\) and therefore the separation property of
\(\PpP'\) combined with \eqref{eqn:aux31} yields that this net converges to
the neutral element (in the compact-open topology).\par
Now we pass to (B). Let \(\QqQ\) be the collection of all continuous proper
quasi-\NORM[s] on \(G\). We claim that:
\begin{enumerate}[(Q1)]
\item \(\QqQ\) induces the topology of \(G\); and
\item \(\QqQ\) is a metric structure.
\end{enumerate}
To show (Q1), we argue as follows. For each compact subgroup \(K\) of \(G\) such
that \(S_K \df A/K\) is Lie, denote by \(\pi_K\dd G \to S_K\) the quotient
homomorphism. Further, for any \(g \in G\) there is a \(\sigma\)-compact open
subgroup \(P(K,g)\) of \(S_K\) such that \(\pi_K(g) \in P(K,g)\). Now let
\(\rho_{K,g}\) be any compatible proper \NORM{} on \(P(K,g)\). Then
a quasi-\NORM{} \(q_{K,g}\) on \(G\) given by
\[q_{K,g}(x) = \begin{cases}\rho_{K,g}(\pi_K(x)) & \UP{if } \pi_K(x) \in
P(K,g)\\\infty & \UP{otherwise}\end{cases}\]
belongs to \(\QqQ\). The family of all \(q_{K,g}\) constructed above is
separating and induces the topology of \(G\)---both these properties follow from
the following well-known property of all LCA groups (consult, e.g.,
\cite[Theorem~1.4.19]{tao}):
\begin{quote}\itshape
Each neighbourhood of the neutral element of a LCA group \(A\) contains
a compact subgroup \(K\) such that the quotient group \(A / K\) is Lie.
\end{quote}
(A detailed verification of the property of the family of all \(q_{K,g}\)
postulated above is left to the reader.) So, (Q1) holds. Moreover:
\begin{itemize}
\item for any \(g \neq e_G\) one may find \(K\) suct that \(0 < q_{K,g}(g) <
 \infty\);
\item if \(u, v \in \QqQ\), then \(\max(u,v) \in \QqQ\) as well.
\end{itemize}
Thus, to check (Q2) it is sufficient to show that \(\QqQ\) is downward directed.
To this end, fix \(u,v \in \QqQ\) and define \(z\dd G' \to [0,\infty]\) and
\(w\dd G \to [0,\infty]\) by \(z = \max(u',v')\) and \(w(g) =
z^{\#}(\kappa_G(g))\). It follows from \LEM{qv-proper} that \(z\) is
a quasi-\NORM{} on \(G'\) that is both continuous and proper. So, another
application of that lemma yields that \(w \in \QqQ\). But \(w \leq u\) and
\(w \leq v\), since \(z \leq u'\) and \(\reg{u} \leq u\) and, similarly,
\(z \leq v'\) and \(\reg{v} \leq v\). This concludes the proof of
(Q1)--(Q2).\par
Now to finish the entire proof of (B), it is sufficient to apply part (A) and
\LEM{qv-proper} to conclude that \(\QqQ'\) is both a proper and reflexive metric
structure on \(G'\). Consequently, \(\PpP \df \QqQ''\) is a metric structure we
searched for in (B).\par
Finally, part (C) easily follows from part (A) and from \LEM{qv-proper}.
\end{proof}

\section{Concluding remarks and open problems}\label{sec:fin}

Although we have managed to adopt the concept of metric duality (naturally
present in the theory of Banach spaces) to the context of LCA groups, at this
stage of our investigations we do not know if there exist reflexive metric
Abelian groups that are, in a sense, far from both LCA groups and Banach spaces.
To be more precise, we put the following

\begin{prb}{1}
Find a reflexive metric Abelian group that is Polish and non-iso\-morphic to
the product of a LCA group and a Banach space.
\end{prb}

The next problem is inspired by a similar well-known property of Banach spaces.

\begin{prb}{2}
Is it true that the topological weight of \((G'',\cdot,p'')\) is not less than
the topological weight of \((G',\cdot,p')\) for an arbitrary \NORM{} \(p\) on
an Abelian group \(G\)?
\end{prb}

Taking into account \THM{main-refl}, the following question naturally arises.

\begin{prb}{3}
Is it true that if \(\PpP\) and \(\QqQ\) are two reflexive metric structures on
an Abelian group \(G\) that induce the same topology on \(G\), then
the collections \(\PpP'\) and \(\QqQ'\) induce the same topology on \(G'\)?
\end{prb}

A special case of the above problem is formulated below.

\begin{prb}{4}
Is it true that if \(\PpP\) is a reflexive metric structure on a LCA group
\(G\), then the collection \(\PpP'\) induces the Pontryagin's topology of
\(G'\)?
\end{prb}

\end{document}